\documentclass{amsart}
\usepackage{amsmath,amssymb,amsfonts,enumerate,amsthm, amscd}
\newtheorem{thm}{Theorem}[section]
\newtheorem{cor}[thm]{Corollary}
\newtheorem{lem}[thm]{Lemma}

\newtheorem{conj}[thm]{Conjecture}
\begin{document}
\bibliographystyle{amsplain}
\title[groups of prime power order]{Non-commuting graphs of  nilpotent groups}
\author{Alireza Abdollahi}
\address{Department  of Mathematics, University of Isfahan, Isfahan 81746-73441, Iran;  \\
School of Mathematics, Institute for Research in Fundamental
Sciences (IPM), P.O.Box 19395-5746 , Tehran, Iran.}
\email{a.abdollahi@math.ui.ac.ir}
\author{Hamid Shahverdi}
\address{School of Mathematics, Institute for research in fundamental science (IPM), P.O. Box 19395-5746, Tehran, Iran.
Iran.}
\email{hamidshahverdi@gmail.com}
\begin{abstract}
Let $G$ be a non-abelian group and $Z(G)$ be the center of $G$.
The non-commuting graph $\Gamma_G$ associated to $G$ is
the graph whose vertex set is $G\setminus Z(G)$ and  two distinct elements  $x,y$ are adjacent  if and only if $xy\neq yx$.
We prove that if $G$ and $H$ are non-abelian nilpotent groups with irregular isomorphic non-commuting graphs, then $|G|=|H|$.
\end{abstract}
\maketitle
\section{\bf Introduction and results}
 Let $G$ be a non-abelian group and $Z(G)$ be its center. The non-commuting graph $\Gamma_G$ of $G$ is a graph whose vertex set  is $G\setminus Z(G)$ and two vertices $x$ and $y$ are adjacent if and only if $xy\neq yx$. The non-commuting
graph of a group was first considered by Paul Erd\H{o}s in 1975 \cite{N}. Many people have  studied the non-commuting graph (e.g., \cite{AADS,AAM,D,SW,WS}).
In \cite{AAM} the following conjecture was put forward:
\begin{conj}[Conjecture 1.1 of \cite{AAM}]\label{con1}
Let $G$ and $H$ be two finite non-abelian groups such that $\Gamma_G\cong \Gamma_H$. Then $|G|=|H|$.
\end{conj}
Conjecture \ref{con1} was refuted by an example due to Isaacs in \cite{M}, however it is valid whenever one of $G$ or $H$ is a non-abelian
finite simple group \cite{D} or whenever one of $G$ or $H$ has  prime power order \cite{AADS}. The counterexample given in \cite{M} is a pair $(G,H)$
of nilpotent non-abelian groups with regular non-commuting graph; recall that a graph is called regular if the degree of all vertices are the same,
otherwise the graph is called irregular.
It follows from  a result of Ito \cite{Ito} that  a finite group with a regular non-commuting graph is a direct product of a non-abelian $p$-group
for some prime $p$ and an abelian group.\\

Here we study  pairs $(G,H)$ of non-abelian finite groups which provide a counterexample to Conjecture \ref{con1}.
It follows from the main result of  \cite{AADS}, that if  a pair $(G,H)$ provides a counterexample then  none of $G$ and $H$ are of  prime power order.
Here we prove that if a pair of non-abelian finite nilpotent groups  provides a counterexample for Conjecture \ref{con1} then their
non-commuting graphs must be regular.
\begin{thm}\label{thm1} Let $G$ and $H$ be two finite non-abelian  nilpotent groups
with irregular non-commuting graphs such that $\Gamma_G\cong\Gamma_H$. Then $|G|=|H|$.
\end{thm}
We conjecture that the word ``nilpotent" in  Theorem \ref{thm1} is sufficient for one of the groups $G$ and $H$.

\section{\bf Non-commuting graphs of nilpotent groups}
A non-abelian group is called an $AC$-group if the centralizer of every non-central element is abelian.
For a group $G$ and an element $g\in G$,  $g^G$ denotes the conjugacy class of $g$ in $G$.
\begin{lem}
Let $G$ and $H$ be two finite non-abelian groups. If $\phi:\Gamma_G\rightarrow \Gamma_H$ is a graph isomorphism and $g$
is a non-central element of $G$, then the following hold:
\begin{enumerate}
\item $|G|-|Z(G)|=|H|-|Z(H)|$.
\item $|G|-|C_G(g)|=|H|-|C_H(\phi(g))|$.
\item $|C_G(g)|-|Z(C_G(g))|=|H|-|Z(C_H(\phi(g)))|$, where $C_G(g)$ is not abelian.
\item If $C_G(g)$ is not abelian, then $\Gamma_{C_G(g)}\cong \Gamma_{C_H(\phi(g))}$.
\item Suppose that $C_1=C_G(g_1)$ and $C_i=C_{C_{i-1}}(g_i)$ for $i\geq 2$, where $g_1\in G\setminus Z(G)$ and $g_i\in C_{i-1}\setminus Z(C_{i-1})$.
Then there exists $k\in\mathbb{N}$ such that $C_k$ is an $AC$-group.
\item $|G|=|H|$ if and only if $|C_G(g)|=|C_H(\phi(g))|$ if and only if $|Z(G)|=|Z(H)|$.
\end{enumerate}
\end{lem}
\begin{proof}
It is straightforward. To prove (5), note that if the centralizer $C_i$ is not an $AC$-group, then some proper centralizer in $C_i$ is not abelian
guaranteeing  the existence of an element $g_{i+1}$. On the other hand, $G$ is finite so the series $C_1>C_2>\cdots>C_i>\cdots$ will  eventually terminate  in
an $AC$-group.
\end{proof}
\begin{lem}[see e.g. Theorem 2.1 of \cite{AADS}] \label{l1} Let $G$ be a finite non-abelian group and $H$ be a group such that
$\Gamma_G\cong \Gamma_H$.  Then the following hold:
\begin{enumerate}
\item  $|C_H(h)|$ divides $(|g^G| - 1) (|Z(G)| - |Z(H)|)$ for any $g\in G\setminus Z(G)$ and $h=\phi(g)$, where $\phi$ is any graph isomorphism
from $\Gamma_G$ to $\Gamma_H$.  \item If $|Z(G)|\geq |Z(H)|$ and $G$ contains a
non-central element  $g$ such that ${|C_G(g)|}^2\geq |G|\cdot|Z(G)|$, then $|G|=|H|$.
\end{enumerate}
\end{lem}
We need the following  result concerning a number theoretic conjecture due to Goormaghtigh.
\begin{thm}[see e.g. Theorem 1.3 of \cite{HT}] \label{l3}
 Let $x,y,m,n$ be integers such that $y>x>1$ and $m,n>1$. Then the following equation has at most one pair $(m,n)$ of solution for every fixed pair $(x,y)$:
 $$\frac{y^n-1}{y-1}=\frac{x^m-1}{x-1}.$$
 \end{thm}
\begin{thm} \label{t1} Let $G$ be a nilpotent group with at least two distinct non-abelian Sylow subgroups. Suppose also that $H$ is any non-abelian group such that
$|Z(G)|\geq|Z(H)|$ and $\Gamma_G\cong\Gamma_H$. Then $|G|=|H|$.
\end{thm}
\begin{proof} Suppose  $G=P\times Q\times S$, where $P$ and $Q$ are non-abelian Sylow $p$, $q$-subgroups of $G$ such that $p\neq q$ and
 $S$ is a subgroup of $G$. If
$x\in P\setminus Z(P)$ and $y\in Q\setminus Z(Q)$, then
$|Z(G)|< |C_P(x)||C_Q(y)||S|$. Therefore
$$|G||Z(G)|< |C_P(x)||C_Q(y)||P||Q||S|^2=|C_G(x)||C_G(y)|.$$ It follows that
 $$|G||Z(G)|< \max\{|C_G(x)|^2,|C_G(y|^2\}.$$ Now,  Lemma \ref{l1}(2) completes the proof.
\end{proof}
\begin{cor} \label{cor1} Let $G$ and $H$ be two nilpotent groups each of which have at least two non-abelian Sylow subgroups.
 If $\Gamma_G\cong\Gamma_H$, then $|G|=|H|$.
\end{cor}
Both groups $G$ and $H$ in the counterexample of Conjecture \ref{con1} due to Isaacs in \cite{M}   have the same shape, that is, they are  direct products
of a non-abelian  group of prime power order  $P$ and a non-trivial abelian group $A$ such that $\gcd(|P|,|A|)=1$ and all non-trivial
conjugacy class sizes of $G$ or $H$ have equal order.
 The latter property was first studied by Ito \cite{Ito} and  we want to prove Theorem \ref{thm1} for all nilpotent groups except
 those satisfying the latter shape.
\section{\bf Proof of Theorem \ref{thm1}}
Now, we prove Theorem \ref{thm1} in four  cases. In this section $G$ and $H$ are finite non-abelian nilpotent groups with irregular non-commuting graphs and
 $\phi:\Gamma_G\rightarrow \Gamma_H$ is
 a graph isomorphism. By Corollary \ref{cor1}, we may assume that $G$ has exactly one non-abelian Sylow subgroup. If $G$ is of prime power order,
   the main result of \cite{AADS} implies that $|G|=|H|$. Thus we may assume $G=P\times A$, where $P$ is a non-abelian Sylow
 $p$-subgroup of $G$ and $A$ is a non-trivial abelian subgroup whose order is prime to $p$. Also, set $|P|=p^n$ and $|Z(P)|=p^r$.\\

{\bf{Case (a): \; $H=P_1\times B$ for some non-abelian Sylow $p$-subgroup $P_1$ of $H$  and for some non-trivial abelian subgroup $B$ of $H$.}}\\
 We use the following notation:   $|P_1|=p^m$, $|Z(P_1)|=p^s$ and $\phi(g_i)=h_i$, where  $g_1,\dots, g_k$  are non-central elements of $G$ chosen from conjugacy classes of $G$ with pairwise distinct sizes such that $|{g_i}^G|=p^{a_i}$ and $|{h_i}^H|=p^{b_i}$ and $a_1<\dots< a_k$ and $b_1<\dots< b_k$. Notice that $k\geq 2$, since $\Gamma_G$ and $\Gamma_H$ are irregular.\\
 Since $\Gamma_G\cong\Gamma_H$,
\begin{equation} |A|p^r(p^{n-r}-1)=|B|p^s(p^{m-s}-1)\end{equation}
\begin{equation} |A|p^{n-a_i}(p^{a_i}-1)=|B|p^{m-b_i}(p^{b_i}-1)\end{equation}
for every $1\leq i\leq k$.
Equation (1) implies that $r=s$ and equation (2) implies that $n-a_i=m-b_i$.
Since  $\Gamma_G$ is not regular,  graph isomorphism implies that
 \begin{equation} |A|(p^{n-a_1}-p^{n-a_2})=|B|(p^{m-b_1}-p^{m-b_2}).\end{equation} Therefore $|A|=|B|$. Now, equation (2) implies that $a_1=b_1$. Hence $|P|=|P_1|$.\\

{\bf{Case (b): \; $H=P_1\times X$, where   $P_1$ is a non-abelian Sylow $p$-subgroup of $H$  and $X$ is an arbitrary group such that $\gcd(p,|X|)=1$.}}\\
 Suppose $H$ is a minimal
counterexample. Also suppose by way of contradiction that $X$ is a non-abelian group. Then $P_1$ and $X$ are $AC$-group.  Let $x\in X\setminus Z(X)$. Then $C_H(x)=P_1\times B $, where $B\subseteq X$ is an abelian subgroup of $X$. Therefore Case (a)
 implies that $|C_H(x)|=|C_G(\phi^{-1}(x))|$. Since $\Gamma_G\cong\Gamma_H$, we have $|G|=|H|$.  Now, $|G|=|H|$ implies that $|P|=|P_1|$. Set
  $|C_G(\phi^{-1}(x))|=p^{n-\alpha}|A|$ for some integer $1<\alpha<n$. By graph isomorphism, we have
$$(p^n-p^{n-\alpha})|A|=p^n(|X|-|C_X(x)|).$$ The largest $p$-power dividing the right-hand side of the equation is $\geq p^n$
and the left is $p^{n-\alpha}$. This is a contradiction. Hence $X$ is abelian and Case (a) completes the proof.\\

{\bf{ Case (c): \;
  $H=Q_1\times X$, where  $Q_1$ is a Sylow  $q_1$-subgroup of $H$ and $X$ is a non-abelian nilpotent group.}}\\
 If $p=q_1$, then Case (b) completes the proof. We claim that $p\neq q_1$ is not possible.   Let $H$ be a minimal counterexample.
Therefore $Q_1$ and $X$ are  $AC$-groups. By the characterization of $AC$-groups \cite{S},  a nilpotent $AC$-group is a direct product of a non-abelian group of
 prime power order and an abelian group. Therefore $X=Q_2\times B$, where $Q_2$ is a non-abelian $q_2$-group for some prime $q_2$, $B$ is an
  abelian group and $\gcd(|Q_2|,|B|)=1$. Let $h_i\in Q_i\setminus Z(Q_i)$ for $i\in\{1,2\}$. Also, set
$\phi^{-1}(h_i)=g_i$ for $i\in\{1,2\}$ and $|C_G(g_i)|=|A|p^{n-{a_i}}$, where $1<a_i<n$. If $q_2=p$, then again Case (b)
implies that $Q_1\times B$ is an abelian group. This is a contradiction. Therefore  $p\neq q_1,q_2$.

We have $C_H(h_1)=C_{Q_1}(h_1)\times Q_2\times B$ and $C_H(h_2)=Q_1\times C_{Q_2}(h_2)\times B$ and $Z(C_H(h_1))=C_{Q_1}(h)\times Z(Q_2)\times B$ and $Z(C_H(h_2))=Z(Q_1)\times C_{Q_2}(h_1)\times B$. So $Z(H)\subsetneqq Z(C_H(h_i))$. Therefore graph isomorphism implies that $Z(G)\subsetneqq Z(C_G(g_i))$. Set $|Z(C_G(g_i))|=|A|p^{d_i}$ for $i\in\{1,2\}$ and $|Z(G)|=|A|p^r$. It is clear that $d_i>r$. Now, $\Gamma_G\cong\Gamma_H$ and $\Gamma_{C_G(g_i)}\cong\Gamma_{C_H(h_i)}$ for $i\in\{1,2\}$ imply that
\begin{equation} |C_H(h_2)|-|Z(C_H(h_2))|=(|Q_1|-|Z(Q_1)|)|C_{Q_2}(h_2)||B|=|A|(p^{n-a_2}-p^{d_2}) \end{equation}
\begin{equation} |Z(C_H(h_1))|-|Z(H)|=(|C_{Q_1}(h_1)|-|Z(Q_1)|)|Z(Q_2)||B|=|A|(p^{d_1}-p^r) \end{equation}
\begin{equation} |H|-|C_H(h_1)|=(|Q_1|-|C_{Q_1}(h_1)|)|Q_2||B|=|A|(p^n-p^{n-a_1}). \end{equation}
Since $|B|(|Q_1|-|Z(Q_1)|)=|B|(|Q_1|-|C_{Q_1}(h_1)|)+|B|(|C_{Q_1}(h_1)|-|Z(Q_1)|)$,
by equation (5) and (6) the largest $p$-power dividing the right-hand side of the latter equation is $p^r$ and by equation (4) the largest $p$-power dividing the left hand side is $p^{d_2}$. This is a contradiction.\\

{\bf{Case (d): \; $H=Q\times B$, where $Q$ is a non-abelian Sylow $q$-subgroup for some prime $q\neq p$ and $B$ is a non-trivial abelian subgroup.}}\\
Suppose by way of contradiction that $|G|\neq |H|$. Since $\Gamma_G $ is not regular,
there exist $g_1,g_2\in G\setminus Z(G)$ such that $|g_1^G|=p^{a_1}\neq p^{a_2}= |g_2^G|$. Set $|Q|=q^m$, $|Z(Q)|=q^s$, $\phi(g_i)=h_i$ for $i\in \{1,2\}$ and
$|h_i^H|=q^{b_i}$. Since $\Gamma_G\cong \Gamma_H$,
\begin{equation} |A|(p^n-p^r)=|B|(q^m-q^s) \end{equation}
\begin{equation} |A|(p^{n-a_i}-p^r)=|B|(q^{m-b_i}-q^s). \end{equation}
If $u=\gcd(a_1,a_2,n-r)$ and $v=\gcd(b_1,b_2,m-s)$, by considering equations (7) and (8) and taking
greatest common divisors, we have
\begin{equation}
|A|p^r(p^{u}-1)=|B|q^s(q^v-1).
\end{equation}
Now, by dividing  equations  (7) and (9), we have
\begin{equation}\frac{p^{n-r}-1}{p^u-1}=\frac{q^{m-s}-1}{q^v-1}\end{equation}
and by dividing  equations  (8) and (9), we have
\begin{equation}\frac{p^{n-a_i-r}-1}{p^u-1}=\frac{q^{m-b_i-s}-1}{q^v-1}.\end{equation}
Note that it is not possible that  $n-a_1-r=u=n-a_2-r$, since $a_1\neq a_2$.
 Now, Theorem \ref{l3} and equation (10) and (11) yield a contradiction. $\hfill\Box$ \\

\begin{center}{\textbf{Acknowledgments}}
\end{center}
The authors are grateful to the referee for his/her invaluable comments.
The first author was financially supported by the Center of Excellence for Mathematics, University of Isfahan. This research was in
part supported by  grants IPM (No. 91050219) and IPM (No. 91200045).


\begin{thebibliography}{99}
\bibitem{AADS}
A.~Abdollahi, S.~Akbari, H.~Dorbidi and H.~Shahverdi,
\emph{Commutativity pattern of non-abelian $p$-groups determine
their orders}, Comm. Algebra, \textbf{41} (2013) 451-461.
\bibitem{AAM}
A.~Abdollahi, S.~Akbari and H. R.~Maimani, \emph{Non-commuting graph
of a group}, J. Algebra, \textbf{298} (2006) 468-492.
\bibitem{D} M.R. Darafsheh, \emph{Groups with the same non-commuting graph},
Discrete Appl. Math., \textbf{157} no. 4 (2009)  833-837.
\bibitem{HT}
B.~He  and A.~Togbe, \emph{On the number of solutions of Goormaghtigh equation for given $x$ and $y$}, Indag. Mathern. N.S., {\bf 19} no. 1 (2008) 65-72.
\bibitem{Ito}
N.~Ito, \emph{On finite groups with given conjugate types}, Nagoya Math. J., \textbf{6} (1953) 17-28.
\bibitem{M}
A. R. Moghaddamfar,  \emph{About noncommuting graphs}, Siberian Math. J., \textbf{47} no. 5 (2005)  1112-1116.
\bibitem{N}
B.H.~Neumann, \emph{A problem of Paul Erd\H{o}s on groups}, J. Aust. Math. Soc. Ser. A, {\bf 21} (1976) 467-472.
\bibitem{S}
R.~Schmidt, \emph{Zenralisatorverbande endlicher gruppen}, Rend. Sem. Mat. Univ. Padova, \textbf{44} (1975) 55-75.
\bibitem{SW}
R.~Solomon and A. Woldar, \emph{Simple non-abelian groups are characterized by their non-commuting graph}, preprint 2012.
\bibitem{WS}
L. Wang and W. Shi. \emph{A new characterization of $A_{10}$ by its non-commuting graph}, Comm. Algebra, \textbf{36} (2008) 533-540.
\end{thebibliography}
\end{document}